\documentclass[11pt]{article}
\usepackage[T1]{fontenc}
\usepackage[utf8]{inputenc}
\usepackage[a4paper,margin=2.35cm]{geometry}
\usepackage{amsmath,amssymb,amsthm,mathtools}
\usepackage{booktabs,tabularx,array,longtable}
\usepackage{graphicx}
\usepackage{microtype,needspace}
\usepackage{xcolor}
\usepackage{enumitem}
\usepackage[hidelinks]{hyperref}
\setlist{nosep}
\allowdisplaybreaks[3]

\newtheorem{theorem}{Theorem}
\newtheorem{lemma}{Lemma}
\newtheorem{proposition}{Proposition}

\theoremstyle{definition}

\newtheorem{remark}{Remark}

\newcommand{\Va}{V_1}
\newcommand{\Vb}{V_2}
\newcommand{\Vc}{V_3}
\newcommand{\HH}{\mathcal H}

\title{An Exact Solution of the Two-Ball Multi-Look Search Game\\with Three Boxes and Heterogeneous Costs}
\author{Igor Kleiner}
\date{}

\hypersetup{pdfauthor={Igor Kleiner},pdftitle={An Exact Solution of the Two-Ball Multi-Look Search Game with Three Boxes and Heterogeneous Costs},pdfsubject={Revised preprint}}
\begin{document}
\maketitle

\begin{abstract}
We study the multi-look search-cost game in which a Hider distributes two
identical balls among three boxes with arbitrary positive inspection costs.
Each inspection either retrieves one ball or reveals that the box is empty.
The Searcher seeks to recover both balls at minimum worst-case expected cost.
We express the value as the maximum of three explicit rational functions and
construct optimal strategies for both players. Depending on the relative
costs, an optimal Hider distribution is obtained by restricting a product-form
distribution to three, five, or all six placements. The adaptive search
reduces to $72$ deterministic trees with $42$ distinct opening-count profiles.
Matching Searcher strategies are constructed from elementary subgames and
explicit rational mixtures. Algebraic identities prove the nonnegativity of
all mixing probabilities on their respective parameter regions, including the
boundaries. The solution characterizes how the support of an optimal hiding
distribution depends on the two cost ratios.

\end{abstract}
\par\medskip

\section{Introduction}
When several objects must be recovered from a set of possible locations,
each successful inspection changes the search problem that remains.
Consider two identical objects hidden in three boxes. Inspecting a nonempty
box retrieves one object, but does not reveal whether the other is also there.
An unsuccessful inspection rules out that box. If the boxes have different
inspection costs, both the choice of the first box and the response to its
outcome matter. We seek a policy that minimizes the expected total cost in
the worst case over the initial placements.

This is the multi-look search-cost model of Lidbetter and
Lin~\cite{lidbetterlin2019}. They solved it for two boxes and for equal
inspection costs. For general heterogeneous costs they proposed a
restricted product-form pattern for optimal hiding probabilities, but their
four-box examples show that simply choosing placements with the largest
product weights need not work. They also solved the related regret game, which charges only
unsuccessful inspections. Here the costs of retrieving the objects
themselves are part of the payoff.

The problem belongs to the literature on discrete search games, surveyed in
\cite{benkoski1991,hohzaki2016,lidbetter2025review}; see also the books
\cite{gal1980,garnaev2000,alperngal2003,alpernetal2013}. In the single-look
multiple-object game of Lidbetter~\cite{lidbetter2013}, at most one object
occupies a box. Here a second visit to a successful location may be necessary.
Repeated searches for a single target, studied in
\cite{robertsgittins1978,gittinsroberts1979,ruckle1991,clarksonetal2023,bui2024,clarksonlin2025},
and the search for one of several objects~\cite{sharlin1987} involve different
information or stopping rules. Caching and network models
\cite{alpernetal2012,yolmeh2021} provide related settings for searches with
multiple targets.

At fixed costs, finite zero-sum game algorithms (see, for example,
\cite{hellerstein2019}) can determine an optimal strategy numerically. The
parametric problem is different: it asks how the strategies and the value
change throughout the cost domain. With two objects and three boxes we
show that three nested candidate supports suffice, although a change in
costs need not pass through all three regimes. This is the smallest
multi-look case beyond two boxes in which both the number of boxes and
number of targets exceed one. The exact solution tests the restricted product-form conjecture in this case.

Our approach is to enumerate the possible adaptive searches once, obtaining
$72$ trees and $42$ distinct opening-count profiles. We then derive three
lower bounds from the Hider's distributions and construct Searcher strategies
that attain the largest bound for every cost vector. The main theorem gives
the value and the hiding probabilities; the subsequent sections give the
matching search policies, with their algebraic feasibility proofs in the
supplement. This yields a solution throughout the two-dimensional space of cost ratios.

\section{Model and finite reduction}
The Searcher acts on a history of inspection outcomes, but the payoff
against a given placement depends only on how many times each box is
opened. This distinction allows us to replace the decision problem by a
finite matrix while retaining the effect of adaptation.

Let the boxes $A,B,C$ have respective inspection costs $a,b,c$,
known to both players. By relabeling we may assume $a\ge b\ge c>0$.
The Hider places two indistinguishable balls before the search begins;
neither ball is moved by the Hider afterwards. The pure Hider placements are
\[
\HH=\{002,011,020,101,110,200\},
\]
where, for example, $101$ means one ball in the $a$-box, none in the
$b$-box, and one in the $c$-box. An inspection of a nonempty box removes exactly one ball; an inspection
of an empty box establishes that it contains no balls. The Searcher observes
success or failure and may condition subsequent inspections on the full
history. A successful inspection does not reveal whether more balls remain
in that box. The Hider maximizes and the Searcher minimizes the expected
sum of inspection costs up to recovery of both balls. A search that fails to
recover both is assigned infinite cost. The value is positively homogeneous: for every $\lambda>0$,
\[
V(\lambda a,\lambda b,\lambda c)=\lambda V(a,b,c).
\]
Thus only two cost ratios determine the form of an optimal strategy.

A policy that inspects a box already known to be empty or continues after recovery is weakly dominated by deleting those inspections. Premature termination can likewise be replaced by a continuation to recovery. We call a policy non-wasteful if it stops at recovery and never inspects a box known to be empty. It suffices to consider such policies that continue to recovery. For a deterministic policy $s$, let
\[
\mathbf N_s(x)=(N_a(x),N_b(x),N_c(x))
\]
denote the numbers of openings of the three boxes against placement $x$.
The scalar search cost is
\[
C_s(x)=aN_a(x)+bN_b(x)+cN_c(x).
\]

\begin{proposition}[Elementary policy reduction]\label{prop:72}
Every non-wasteful deterministic Searcher policy is payoff-equivalent to one
of $72$ elementary decision trees. These $72$ trees generate exactly $42$
distinct six-placement opening-count profiles.
\end{proposition}

\begin{proof}
Choose the first box in $3$ ways. If the first opening succeeds, exactly one
ball remains. A non-wasteful deterministic continuation is therefore an
ordering of the three boxes, giving $3!=6$ choices: a failed opening removes a
box from consideration, and the process stops at the first success.

If the first opening fails, both balls lie in the other two boxes, say $j$ and
$k$. The Searcher chooses which of them to open next. After success she may
either repeat that box or switch to the other; after failure the other box is
known to contain both balls. Thus there are $2\cdot2=4$ non-wasteful two-box continuations. Hence there are $3\cdot6\cdot4=72$ elementary trees.

To count distinct profiles, first fix the initial box. The six permutations
on the success branch give distinct profiles for the three placements with a
ball in that box. On the failure branch the two rules that switch boxes after
a success have identical profiles; the two repeat rules give two further
profiles. Thus each initial box gives $6\cdot3=18$ distinct profiles.
For a profile to be shared by two initial boxes, say $A$ and $B$, both must
be inspected under every placement. After a successful first inspection of
$A$, the continuation permutation must begin with $B$ (two choices). After a
failure of $A$, two of the three continuation profiles always inspect $B$.
Both can be represented by opening $B$ next. The resulting four profiles can
also start with $B$, by reversing the first two inspections and retaining
the subsequent rule for each pair of outcomes. No profile belongs to all three initial-box classes: for any initial box, a success followed by a success in
the first box of its continuation permutation omits the third box.
Inclusion--exclusion therefore gives $3\cdot18-3\cdot4=42$. The supplement
lists these profiles explicitly.
\end{proof}

We may represent a randomized search by drawing a deterministic tree
before the first inspection: its contingent choices encode all the
randomized decisions that may be needed. The resulting zero-sum matrix
has six Hider rows and $42$ Searcher columns. Let $\mathcal S$ be the set
of these profiles and $\Delta(K)$ the probability simplex on a finite set
$K$. The finite minimax theorem gives
\[
V(a,b,c)=\min_{q\in\Delta(\mathcal S)}\max_{x\in\HH}
    \sum_{s\in\mathcal S}q_s C_s(x)
=\max_{p\in\Delta(\HH)}\min_{s\in\mathcal S}
    \sum_{x\in\HH}p_x C_s(x).
\]
Different trees can give the same opening-count vector for every placement.
Such trees have the same payoffs for all costs and may be represented by a
single column. Further coincidences of scalar costs at particular cost
vectors do not affect this reduction. Appendix~\ref{app:policies} describes
the $14$ named policies used below.

The finite matrix is now fixed in terms of opening counts, with entries
linear in the costs. We next identify three hiding distributions whose
guarantees determine its value.

\section{Value candidates and nested Hider supports}
A full product-form hiding distribution equalizes non-wasteful searches
\cite{lidbetterlin2019}, but it need not be optimal when costs differ. This
suggests two restrictions: force a ball into the most expensive box, or
exclude only the placement with both balls in the cheapest box. We retain
the original product weights on the allowed placements. These two
restrictions, together with the full distribution, lead to the following
three candidates for the value.

For $m=1,2,3$, let
\[
T_m=\sum_{\substack{x_1+x_2+x_3=m\\x_i\in\mathbb Z_{\ge0}}}a^{x_1}b^{x_2}c^{x_3}
\]
be the complete homogeneous polynomial of degree $m$ in $a,b,c$. Explicitly,
\begin{align*}
T_1&=a+b+c,\\
T_2&=a^2+b^2+c^2+ab+ac+bc,\\
T_3&=a^3+b^3+c^3+a^2b+a^2c+ab^2+b^2c+ac^2+bc^2+abc.
\end{align*}
Define three nested sets of Hider placements,
\begin{align*}
\HH_1&=\{101,110,200\},\\
\HH_2&=\HH_1\cup\{011,020\},\\
\HH_3&=\HH_2\cup\{002\}=\HH.
\end{align*}
For $i=1,2,3$, let $h_i$ be the product-form distribution on $\HH_i$:
the probability of a placement $x=(x_1,x_2,x_3)$ is proportional to
$a^{x_1}b^{x_2}c^{x_3}$. In the row order
$(002,011,020,101,110,200)$, these distributions are
\begin{align*}
h_1&=\frac1{T_1}(0,0,0,c,b,a),\\
h_2&=\frac1{T_2-c^2}(0,bc,b^2,ac,ab,a^2),\\
h_3&=\frac1{T_2}(c^2,bc,b^2,ac,ab,a^2).
\end{align*}
The common factor $a$ has been cancelled in the expression for $h_1$.
Thus $h_1$ forces at least one ball into the most expensive box, $h_2$
excludes only placement $002$, and $h_3$ uses all six placements.

The corresponding guaranteed values are
\begin{align}
\Va&=a+\frac{T_2}{T_1},\label{eq:V1}\\
\Vb&=\frac{2T_3-c^2(a+b+2c)}{T_2-c^2},\label{eq:V2}\\
\Vc&=\frac{2T_3}{T_2}.\label{eq:V3}
\end{align}
The term subtracted in the numerator of $V_2$ is the product weight of
placement $002$ times its largest possible non-wasteful search cost,
$a+b+2c$. The next section establishes these three lower bounds, and the
matching upper bounds then show that their maximum is the value.

\begin{theorem}[Value and optimal Hider strategies]\label{thm:main}
For all $a\ge b\ge c>0$,
\[
\boxed{V(a,b,c)=\max\{\Va,\Vb,\Vc\}.}
\]
Whenever $\Va$, $\Vb$, or $\Vc$ is maximal, $h_1$, $h_2$, or $h_3$,
respectively, is an optimal Hider strategy. If several candidates tie for the
maximum, every convex combination of the corresponding distributions is also
optimal.
\end{theorem}

The three exhibited supports form a nested family:
\[
|\HH_1|=3,\qquad |\HH_2|=5,\qquad |\HH_3|=6.
\]
Figure~\ref{fig:regions} displays the three value regions after the scale-free
normalization $r=b/a$ and $u=c/b$.

\begin{figure}[!htb]
\centering
\includegraphics[width=.67\textwidth]{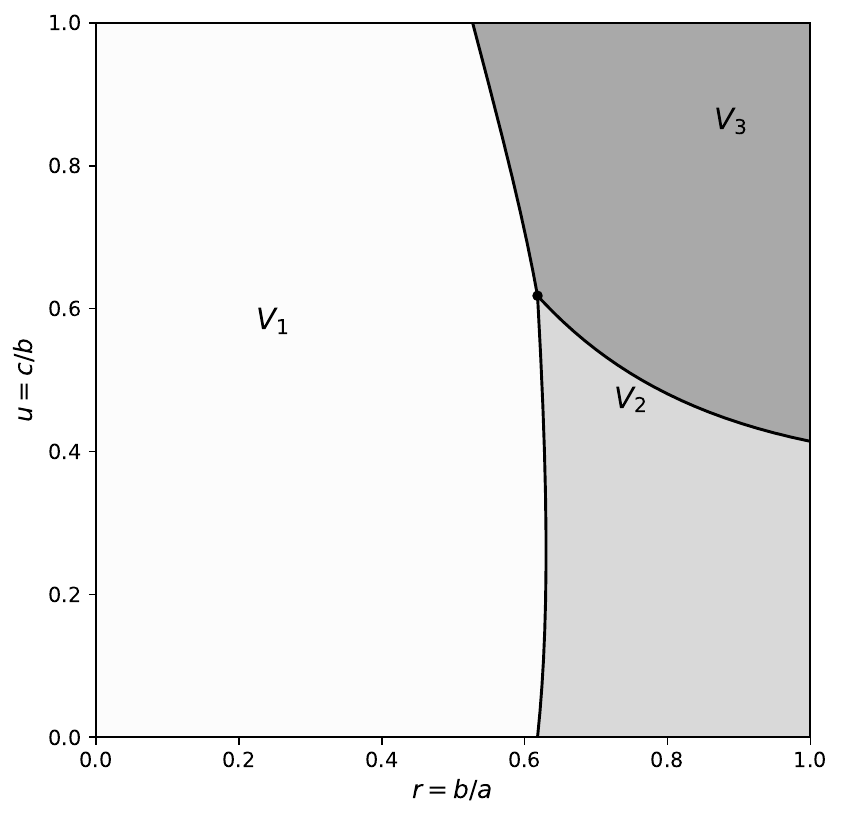}
\caption{Dominant candidate values in the normalized cost domain
$(r,u)\in(0,1]^2$. Solid curves separate regions where different
candidates attain the maximum. The three-way tie occurs at
$r=u=(\sqrt5-1)/2$. The axes at zero represent limiting costs, not
admissible cost vectors.}
\label{fig:regions}
\end{figure}

\begin{remark}[A direct transition when $b=c$]\label{rem:direct}
The nested supports do not imply that every continuous change of costs passes
through all three regimes. For $u=c/b=1$, the conditions for $V_2$ to be
maximal reduce to
\[
4r^2+r-2\ge0,\qquad 1-2r-r^2\ge0,
\]
which cannot hold simultaneously for $0<r\le1$. Indeed, the first inequality
requires $r\ge(\sqrt{33}-1)/8>1/2$, whereas the second requires
$r\le\sqrt2-1<1/2$. The candidate maximum switches directly from $V_1$
to $V_3$ at
\[
r=\frac{\sqrt{22}-1}{7},\qquad
\frac ab=\frac{\sqrt{22}+1}{3}.
\]
On the switching boundary both of the corresponding Hider distributions
are optimal by Theorem~\ref{thm:main}.
\end{remark}

\begin{table}[!htb]
\centering
\caption{Three cost vectors realizing the three value regimes.}
\label{tab:examples}
\begin{tabular}{@{}cccc@{}}
\toprule
$(a,b,c)$ & Maximal candidate & Exact value & Support size\\
\midrule
$(5,1,1)$ & $V_1$ & $73/7$ & $3$\\
$(3,2,1)$ & $V_2$ & $173/24$ & $5$\\
$(1,1,1)$ & $V_3$ & $10/3$ & $6$\\
\bottomrule
\end{tabular}
\end{table}

Table~\ref{tab:examples} gives a representative cost vector for each
regime. Remark~\ref{rem:direct} illustrates why both cost ratios are needed:
the intermediate regime can be bypassed. We now prove the value formula
by bounding the game from both sides.

\section{Hider lower bounds}
The full product-form distribution gives an equalizing cost and also
helps evaluate its restriction $h_2$. For $h_1$, we condition on the first
inspection. The remaining problem involves one ball or two boxes.
We will also use the values of these elementary games for the upper
bound in the first regime. Write
\begin{align*}
U_3(a,b,c)&=\frac{a^2+b^2+c^2+ab+ac+bc}{a+b+c},\\
U_2(x,y)&=\frac{x^2+xy+y^2}{x+y},\\
J(x,y)&=x+U_2(x,y),\\
E(x,y)&=\frac{2(x^3+x^2y+xy^2+y^3)}{x^2+xy+y^2}.
\end{align*}
Here $U_3$ and $U_2$ are the minimax values for one ball, also the
expected costs against the corresponding product-form hiding distributions.
The quantity $J(x,y)$ is the minimum expected cost when one ball is
guaranteed to be in the $x$-box and the other is distributed proportionally
to $x,y$: opening the $x$-box leaves a one-ball problem, and opening the
$y$-box first gives the same expected cost. Under the full two-ball
product-form distribution in two boxes, every non-wasteful policy has
expected cost $E(x,y)$~\cite{lidbetterlin2019}. Thus $J$ and $E$ describe costs against specified hiding distributions;
the unrestricted two-box minimax value is used in the next section.

\begin{lemma}\label{lem:lower}
For every non-wasteful Searcher policy $s$,
\[
\mathbb E_{h_1}[C_s]\ge\Va,
\qquad
\mathbb E_{h_2}[C_s]\ge\Vb,
\qquad
\mathbb E_{h_3}[C_s]=\Vc.
\]
For an arbitrary Searcher policy, all three expectations are at least the
corresponding displayed values.
\end{lemma}
\begin{proof}
Start with the full product-form distribution $h_3$. By the equalizing
lemma of~\cite{lidbetterlin2019}, every non-wasteful deterministic policy
has expected cost $2T_3/T_2=\Vc$ against $h_3$.

The distribution $h_2$ is $h_3$ conditioned on excluding $002$. When both
balls occupy $C$, a non-wasteful search opens $C$ twice and each of $A,B$
at most once. Hence $C_s(002)\le a+b+2c$. Conditioning the full product-form
expectation gives
\[
\mathbb E_{h_2}[C_s]
=\frac{2T_3-c^2 C_s(002)}{T_2-c^2}
\ge\frac{2T_3-c^2(a+b+2c)}{T_2-c^2}=\Vb.
\]
The bound is attained against $h_2$ by a policy whose cost at $002$ is
$a+b+2c$, for example one that inspects $A$ and then $B$ after failure
before searching $C$.

Under $h_1$, a first inspection of $A$ certainly succeeds. A first
inspection of either cheaper box instead reveals whether the second ball
is there. The conditional costs yield the following lower bounds, one for
each possible first choice:
\begin{align*}
A:&\quad a+U_3(a,b,c)=\Va,\\
B:&\quad b+\frac{b}{T_1}a+\frac{a+c}{T_1}J(a,c)=\Va,\\
C:&\quad c+\frac{c}{T_1}a+\frac{a+b}{T_1}J(a,b)=\Va.
\end{align*}
For example, after a successful first search of $B$, the remaining ball is
known to be in $A$; after a failed first search of $B$, the conditional
problem is the forced-one two-box game on $A,C$. Thus every first action has
conditional value at least $\Va$.

Finally, deleting redundant inspections or extending a prematurely
terminated search to recovery cannot increase its cost. This proves the
lower bounds for arbitrary policies as well.
\end{proof}
By linearity, Lemma~\ref{lem:lower} also applies to mixed Searcher strategies,
so
\[
V(a,b,c)\ge\max\{\Va,\Vb,\Vc\}.
\]
The Hider can therefore guarantee the largest candidate. It remains to
construct, for each region, a Searcher strategy with expected cost at most
that candidate against every placement.

\section{Upper bound in the \texorpdfstring{$V_1$}{V1} regime}
The strategy in the first regime follows the structure of $h_1$: start
with the most expensive box. If it is empty, the problem becomes a
two-box game. The key is to show that this failure branch is no more
expensive than the continuation after a success.

Assume $\Va\ge\Vb$ and $\Va\ge\Vc$, and open the $a$-box first. After a success,
use an optimal one-ball strategy, which guarantees additional expected
cost at most $U_3(a,b,c)$ against any remaining placement. After a failure,
use an optimal strategy for the two-ball game on costs $b,c$. By the
two-box result of Lidbetter and Lin~\cite{lidbetterlin2019}, its value is
\[
\max\{J(b,c),E(b,c)\}.
\]
To compare the two continuations, observe that the differences of the
candidate values factor as
\begin{align}
\Va-\Vb
&=\frac{b(b+c)}{a^2+ab+ac+b^2+bc}
  \bigl(U_3(a,b,c)-J(b,c)\bigr),\label{eq:factor12}\\
\Va-\Vc
&=\frac{b^2+bc+c^2}{T_2}
  \bigl(U_3(a,b,c)-E(b,c)\bigr)\label{eq:factor13}
\end{align}
Since $\Va$ is maximal, both two-box candidates are at most $U_3(a,b,c)$. The Searcher
therefore guarantees
\[
a+U_3(a,b,c)=\Va.
\]
This matches the guarantee of $h_1$ and proves the value in the first
regime. In particular the first inspection of an optimal search may be
deterministic here, even though later inspections can be randomized. The
next regime calls for a different construction.

\section{Upper bound in the \texorpdfstring{$V_2$}{V2} regime}
In the second regime, $h_2$ assigns positive probability to five placements.
A matching Searcher strategy must give cost $V_2$ at each of them and no
more at $002$. We obtain mixing probabilities by solving these payoff
equalities together with normalization for selected sets of policies. Two
sets suffice; their use is determined by the signs of the polynomials
below.

Normalize $(a,b,c)=(1,r,ru)$ with $0<r,u\le1$, and define
\begin{align*}
A_2&=r^2(1+u)^2+r-u-1,\\
K_2&=1-r(1+u)+r^2(1-u-u^2),\\
Q_2&=-2r^3u^3-4r^3u^2-2r^3u-r^2u^3-3r^2u^2-3r^2u+ru+u+2.
\end{align*}
The exact identities
\begin{align*}
\Vb-\Va&=\frac{r^2A_2}{(1+r+ru)(r^2u+r^2+ru+r+1)},\\
\Vb-\Vc&=\frac{r^2u^2(1+r)K_2}
{(r^2u+r^2+ru+r+1)(r^2u^2+r^2u+r^2+ru+r+1)}
\end{align*}
show that the $V_2$ region is exactly $A_2\ge0$, $K_2\ge0$. The additional
polynomial $Q_2$ selects a feasible mixture; both constructions are
available when it vanishes.

\begin{lemma}[Searcher certificate in the $V_2$ region]\label{lem:V2cert}
Suppose $A_2\ge0$ and $K_2\ge0$. If $Q_2\le0$, there is a Searcher
mixture using only the policies
\[
\{26,27,33,38,39\};
\]
if $Q_2\ge0$, there is a mixture using only the policies
\[
\{26,33,37,38,41\}.
\]
In both cases the mixture has nonnegative weights summing to one and guarantees
expected cost at most $\Vb$ against every Hider placement.
\end{lemma}
\begin{proof}
Use the weights defined in the supplement, with the named policies of
Appendix~\ref{app:policies}. Substituting their rational expressions into
the payoff matrix gives cost $\Vb$ at every placement in $\HH_2$ and
unit total weight. For the only remaining placement, let $\overline C(x)$
denote the expected cost of the mixture; direct simplification gives
\[
\overline C(002)-\Vb
=-\frac{(a+b)(a^2-ab-ac+b^2-bc-c^2)}{a^2+ab+ac+b^2+bc}\le0,
\]
because $\Vb\ge\Vc$ is equivalent to nonnegativity of the second factor.
It follows that the cost bound holds against all six placements. The
nonnegativity of the weights is established in
Section~\ref{sec:algebraic}, using the explicit identities in the supplement.
\end{proof}
The two branches give the same value, even though their search mixtures
are different. The excluded placement has a separate cost bound. In the
third regime this distinction disappears, since $h_3$ has full support.

\section{Upper bound in the \texorpdfstring{$V_3$}{V3} regime}
When $h_3$ is optimal, every placement has positive hiding probability.
The Searcher must therefore equalize the costs of all six placements.
As in the preceding section, we solve the resulting linear equations
for selected policy sets. Here four sets cover the region. First, the
comparisons with the other two candidate values give the polynomials
\begin{align*}
A_3&=r^2u^4+2r^2u^3+r^2u^2+2r^2u+r^2+ru^3+r-u^2-u-1,\\
B_3&=r^2u^2+r^2u-r^2+ru+r-1.
\end{align*}
Notice that $B_3=-K_2$: both regimes therefore use the same equation for
the boundary $V_2=V_3$. The exact identities
\begin{align*}
\Vc-\Va&=\frac{r^2A_3}{(1+r+ru)(r^2u^2+r^2u+r^2+ru+r+1)},\\
\Vc-\Vb&=\frac{r^2u^2(1+r)B_3}
{(r^2u+r^2+ru+r+1)(r^2u^2+r^2u+r^2+ru+r+1)}
\end{align*}
show that the $V_3$ region is exactly $A_3\ge0$, $B_3\ge0$. Within this
region three further polynomials give sufficient sign conditions for
choosing among the proposed mixtures:
\begin{align*}
X&=r^2u^2+3r^2u+r^2-2ru^2-ru+r-1,\\
Y&=-2r^4u^3-6r^4u^2-2r^4u+2r^4+2r^3u^3-r^3u^2-5r^3u-r^3\\
&\quad{}+r^2u^2+2r^2u+2r^2+ru+2r+1,\\
Z&=r^2u^2+r^2u+r^2-ru+r-1.
\end{align*}
\begin{lemma}[Searcher certificates in the $V_3$ region]\label{lem:V3cert}
Suppose $A_3\ge0$ and $B_3\ge0$. The following four cases cover the region,
and in each case there is a Searcher mixture using only the indicated
policies. Its weights are nonnegative and sum to one, and its expected
cost is $\Vc$ against each of the six Hider placements.
\begin{center}
\begin{tabular}{@{}ccl@{}}
\toprule
Case & Policy support & Additional conditions\\
\midrule
A & $\{2,26,28,33,37,40\}$ & $Y\ge0$, $X\ge0$\\
B & $\{2,3,22,26,37,40\}$ & $Y\le0$\\
C & $\{2,5,28,33,37,40\}$ & $Y\ge0$, $X\le0$, $Z\ge0$\\
D & $\{5,19,28,33,37,40\}$ & $Y\ge0$, $X\le0$, $Z\le0$\\
\bottomrule
\end{tabular}
\end{center}
\end{lemma}
\begin{proof}
The alternatives first distinguish the sign of $Y$ and, on the
nonnegative side, the signs of $X$ and then $Z$. They therefore cover the
whole region, with overlaps at zero. Substitution of the weights in the
supplement gives unit total weight and expected cost $\Vc$ for each of
the six placements. Section~\ref{sec:algebraic} proves that these
weights are nonnegative throughout their respective branches.
\end{proof}
The four mixtures thus supply the required cost equalities in the
full-support region. To complete the argument for both the second and third
regimes, we prove that the listed rational weights are legitimate
probabilities.

\section{Algebraic certificates for the Searcher mixtures}\label{sec:algebraic}
The region inequalities do more than identify which value is maximal: they
also provide the ingredients for proving the signs of the mixing weights.
We first use them to obtain a common bound on the cost ratio, then express
each weight with a positive denominator and a nonnegative numerator.
The supplement gives all the identities needed for this step.

\begin{lemma}[A common lower bound on the cost ratio]\label{lem:half}
In either the $V_2$ or the $V_3$ region, $r>1/2$.
\end{lemma}
\begin{proof}
Recall the defining conditions $A_2\ge0$ for regime $V_2$ and
$A_3\ge0$ for regime $V_3$. Both polynomials are strictly increasing in
$r>0$; explicitly,
\[
\partial_r A_2=2r(1+u)^2+1>0,\qquad
\partial_r A_3=2r(u^4+2u^3+u^2+2u+1)+u^3+1>0.
\]
At $r=1/2$ one has
\begin{align*}
A_2(1/2,u)&=-\frac{1+u(2-u)}{4}<0,\\
A_3(1/2,u)&=-\frac{1+2u(1-u)+u^2(1-u^2)
                        +4u^2(1-u)}{4}<0.
\end{align*}
All terms displayed after the constant $1$ in the last numerator are
nonnegative for $0<u\le1$. Monotonicity proves the claim.
\end{proof}

\paragraph{Denominator signs.}
Put $v=2r-1>0$ and let
\[
H=2ru^2+5ru+2r-u,\qquad
G=r^2u^2-r^2u-2r^2-2ru-3r-3.
\]
The two denominators requiring attention satisfy
\begin{align*}
2H &=v(2u^2+5u+2)+2u^2+3u+2>0,\\
-G &=r^2u(1-u)+2r^2+2ru+3r+3>0.
\end{align*}
All remaining denominator factors in the supplement are positive
monomials or sums of positive terms. In particular, positivity of $H$ follows from the bound on $r$ in
Lemma~\ref{lem:half}.

\paragraph{Positive numerators.}
Write $N_j^{(k)}$ for the oriented numerator of the weight assigned to
policy $j$ in branch $k$, with the strictly positive denominator
specified in the supplement. Ten of the $34$ numerator signs follow directly from the branch conditions
or positivity of $u$. Five more follow by short decompositions. Each of
the other $19$ numerators is a sum of nonnegative multiples of branch
polynomials and of $v$ and $H$, with a possible nonnegative remainder.
The multipliers and remainders are polynomials with nonnegative
coefficients in $r,u,1-r,1-u$.

For example, in branch $\mathrm{II}^{-}$, where
$A_2,K_2,-Q_2\ge0$, the following identity is sufficient for one
weight:
\begin{align*}
N_{26}^{(\mathrm{II}^{-})}
  &=r(1+u)^2A_2
   +\bigl[r(1+u-u^2)+3-u^2\bigr]K_2
   +(1-u)(-Q_2)\ge0.
\end{align*}
The coefficient of $K_2$ is nonnegative because
$1+u-u^2\ge1$ and $3-u^2\ge2$.
Similarly, branch $\mathrm{III,A}$ has $B_3,Y,X\ge0$ and
\begin{align*}
N_{37}^{(\mathrm{III,A})}
  &=Y+2(1+r^2)B_3+2r(1-r)X\ge0,\\
N_{40}^{(\mathrm{III,A})}
  &=Y+2[1+r+ru(1-r)]B_3+4r^2uX\ge0.
\end{align*}
The remaining numerator identities in the supplement use the same
principle.

\begin{proposition}[Feasibility of the six mixtures]\label{prop:feasibility}
Each of the six Searcher mixtures of Lemmas~\ref{lem:V2cert}
and~\ref{lem:V3cert} has nonnegative weights summing to one throughout
its stated parameter branch, including the shared boundaries.
\end{proposition}
\begin{proof}
The preceding denominator bounds and the complete positive numerator
identities in the supplement establish nonnegativity. The six
normalization identities follow by rational simplification of the
listed weights; they are reproduced there together with all required
payoff equalities. The branches cover the $V_2$ and $V_3$ regions:
according to the sign of $Q_2$ in the former, and according to the
successive signs of $Y,X,Z$ in the latter. Consequently the upper
bounds in all three regimes match the lower bounds of
Lemma~\ref{lem:lower}, proving Theorem~\ref{thm:main}.
\end{proof}

The result is a complete sign certificate for the six mixtures: the same
expressions that define their probabilities establish their feasibility.
The upper and lower bounds consequently agree over the entire positive
cost domain.

\section{Discussion and conclusions}
The solution separates two aspects of the game: the value is obtained
from three rational functions, while the optimal search policies depend
on the geometry of the corresponding regions. A particularly simple
feature on the Hider side is the nested family of supports
\[
\HH_1\subset\HH_2\subset\HH_3.
\]
In the first region the Hider can guarantee a ball in the most expensive
box. In the second, both balls may avoid that box, but cannot both occupy
the cheapest box under the exhibited strategy. In the third, every
placement has positive probability. These are three alternatives, not a
mandatory sequence: when $b=c$, the second region is absent
(Remark~\ref{rem:direct}). Thus a two-dimensional description of the cost
ratios is essential even for this small game.

The result establishes the restricted product-form property conjectured
in~\cite{lidbetterlin2019} for the two-ball, three-box case. Although
$h_1$ and $h_2$ retain the relative product weights of $h_3$, they need not
equalize all non-wasteful search policies. The appropriate support, as well
as the weight formula, is therefore part of the solution. On the Searcher
side, the proof uses only $14$ named policies from the complete set of $42$
profiles.

The present analysis assumes perfect recovery of one ball at each successful
inspection and a fixed placement. For more boxes or balls, the same finite
reduction remains available, but neither the nested support structure nor
the six mixing constructions can be carried over without further argument.
A natural next question is whether an analogous small family of hiding
distributions suffices for three balls in three boxes or for two balls in
four boxes. The exact value and strategies obtained here provide a benchmark
for such extensions.

\paragraph{Use of AI-assisted tools.}
AI-assisted tools were used in developing and checking the algebraic
certificates, conducting computational checks, and editing the manuscript.
The author takes responsibility for the mathematical claims, proofs,
references, and computational results.

\appendix
\section{Structured Searcher policies used in the certificates}\label{app:policies}
The optimal mixtures in the second and third regimes refer to fourteen
numbered policies. This catalogue specifies their contingent decisions,
so that a mixture determines an implementable search rule rather than
just a vector of costs.

Write $\mathcal P(i\mid\sigma\mid R_{jk})$ or
$\mathcal P(i\mid\sigma\mid S_{jk})$ for the following deterministic policy.
First inspect box $i$. If this succeeds, one ball remains: inspect boxes
in the order $\sigma$ until finding it. If the first inspection fails,
both balls are in the other two boxes $j,k$; inspect $j$ next. If this inspection succeeds, $R_{jk}$ repeats $j$, whereas $S_{jk}$
switches to $k$; if that next inspection fails, inspect the other box
to recover the remaining ball. If the inspection of $j$ fails before any
ball has been found, both balls are in $k$, which is inspected twice.
The search stops as soon as the second ball is recovered.

\begin{center}
\small
\begin{tabular}{@{}rll@{}}
\toprule
Profile name & Deterministic rule & Used in certificates\\
\midrule
2 & $C\mid CBA\mid R_{AB}$ & $V_3$: A, B\\
3 & $C\mid CAB\mid R_{BA}$ & $V_3$: B\\
5 & $C\mid CAB\mid R_{AB}$ & $V_3$: C, D\\
19 & $A\mid ABC\mid R_{CB}$ & $V_3$: D\\
22 & $B\mid CAB\mid S_{AC}$ & $V_3$: B\\
26 & $B\mid BCA\mid R_{AC}$ & $V_2$: both; $V_3$: A, B\\
27 & $C\mid ABC\mid R_{BA}$ & $V_2$: $Q_2\le0$\\
28 & $A\mid CBA\mid S_{BC}$ & $V_3$: A, C, D\\
33 & $A\mid ABC\mid S_{BC}$ & $V_2$: both; $V_3$: A, C, D\\
37 & $A\mid CAB\mid R_{BC}$ & $V_2$: $Q_2\ge0$; $V_3$: all\\
38 & $A\mid BCA\mid R_{BC}$ & $V_2$: both\\
39 & $A\mid ACB\mid R_{BC}$ & $V_2$: $Q_2\le0$\\
40 & $A\mid BAC\mid R_{BC}$ & $V_3$: all\\
41 & $A\mid ABC\mid R_{BC}$ & $V_2$: $Q_2\ge0$\\
\bottomrule
\end{tabular}
\end{center}
For example, $\mathcal P(B\mid BCA\mid R_{AC})$ inspects $B$ first.
After success, it checks $B,C,A$ in that order for the remaining ball.
After failure, it inspects $A$; a success is followed by another inspection
of $A$, whereas a failure identifies $C$ as containing both balls.

The policy names in this table are distinct from the row numbers of the
lexicographic profile catalogue in the supplement. All fourteen policies
belong to the set of $42$ profiles from Proposition~\ref{prop:72}. Their
mixture weights determine which contingent search rule to draw at the
start of the game.

\section{Reproducibility}
The mathematical supplement contains all six mixtures and the complete
set of $42$ profiles. The accompanying code archive provides the weights
and polynomial identities in machine-readable form, together with scripts
for checking the matrix, normalization and payoffs. After extracting the
code archive, run the following commands from its root:
\begin{verbatim}
python -m pip install -r code/requirements.txt
bash code/run_checks.sh
\end{verbatim}
The first command installs the dependencies; the second checks the polynomial identities, regenerates the complete
strategy matrix and verifies the normalization and payoff conditions by
exact arithmetic. The optional script \texttt{code/run\_extended\_checks.sh}
also runs a Bernstein-basis sign check (based on the range-enclosure
methods in~\cite{garloff1986,farouki2012}) and numerical boundary tests.
The expected counts are $72$ trees, $42$ profiles and $34$ verified
weight formulas across the six mixtures. These provide checks on the
explicit constructions used in the proof.

\clearpage
\phantomsection
\pdfbookmark[0]{Mathematical supplement}{supplement}
\begin{center}
{\Large\bfseries Supplement: Explicit Optimal Searcher Mixtures\\[3pt]
and Algebraic Certificates\par}
\end{center}
\setcounter{section}{0}
\setcounter{subsection}{0}
\setcounter{equation}{0}
\setcounter{lemma}{0}
\setcounter{proposition}{0}
\setcounter{table}{0}
\renewcommand{\thesection}{S\arabic{section}}
\renewcommand{\thesubsection}{\thesection.\arabic{subsection}}
\renewcommand{\theequation}{S\arabic{equation}}
\renewcommand{\thelemma}{S\arabic{lemma}}
\renewcommand{\theproposition}{S\arabic{proposition}}
\renewcommand{\thetable}{S\arabic{table}}
\renewcommand{\theHsection}{supp.\arabic{section}}
\renewcommand{\theHsubsection}{\theHsection.\arabic{subsection}}
\renewcommand{\theHequation}{supp.\arabic{equation}}
\renewcommand{\theHlemma}{supp.\arabic{lemma}}
\renewcommand{\theHproposition}{supp.\arabic{proposition}}
\renewcommand{\theHtable}{supp.\arabic{table}}

\noindent The main theorem uses two mixtures in regime II and four in
regime III. This supplement gives their mixing probabilities in a form
that also proves their signs: positive denominators and explicit sums
of nonnegative terms. After recording the branch conditions, we give the
weights and verify the payoff equalities. The final catalogue contains
all $42$ opening-count profiles needed to reconstruct the finite game.

\section{Notation and regions}
The branches specify where each of the proposed mixtures may be used.
The two cost ratios determine these regions and all mixing probabilities.
We collect the common denominator factors and region polynomials before
giving the branchwise expressions.

By homogeneity we may normalize $(a,b,c)=(1,r,ru)$, with $0<r,u\le1$. For general costs, multiply the normalized expected costs by $a$; the weights are unchanged. Write
\begin{align*}
D_2&=1+r+ru+r^2+r^2u,& D_3&=D_2+r^2u^2,\\
L&=1+r+ru,& U&=1+u,\qquad J_0=1+r,\\
H&=2ru^2+5ru+2r-u,&
G&=r^2u^2-r^2u-2r^2-2ru-3r-3.
\end{align*}
The region polynomials are
\begin{align*}
A_2&=r^2(1+u)^2+r-u-1,\\
K_2&=1-r(1+u)+r^2(1-u-u^2),\\
Q_2&=-2r^3u^3-4r^3u^2-2r^3u-r^2u^3-3r^2u^2-3r^2u+ru+u+2,\\
A_3&=r^2(u^4+2u^3+u^2+2u+1)+r(u^3+1)-u^2-u-1,\\
B_3&=-K_2,\\
X&=r^2(u^2+3u+1)+r(-2u^2-u+1)-1,\\
Y&=r^4(-2u^3-6u^2-2u+2)+r^3(2u^3-u^2-5u-1)\\
 &\qquad+r^2(u^2+2u+2)+r(u+2)+1,\\
Z&=r^2(u^2+u+1)+r(1-u)-1.
\end{align*}
The six branches, including their boundaries within the positive cost domain, are
\begin{center}
\begin{tabular}{@{}ll@{}}
\toprule
Branch & Required nonnegative polynomials\\ \midrule
$\mathrm{II}^{-}$ & $A_2,K_2,-Q_2$\\
$\mathrm{II}^{+}$ & $A_2,K_2,Q_2$\\
$\mathrm{III,A}$ & $A_3,B_3,Y,X$\\
$\mathrm{III,B}$ & $A_3,B_3,-Y$\\
$\mathrm{III,C}$ & $A_3,B_3,Y,-X,Z$\\
$\mathrm{III,D}$ & $A_3,B_3,Y,-X,-Z$\\
\bottomrule
\end{tabular}
\end{center}
The two branches of II cover the $V_2$ region, and the four branches of
III cover the $V_3$ region. Each subsequent sign argument uses only the
inequalities in its own row of the table. The first regime is handled
separately in the main paper.

\section{A common domain simplification and denominators}
The two factors $H$ and $G$ in the denominators are not immediately
positive from their expanded forms. The region inequalities yield a
simple lower bound on $r$, from which the required signs follow.

\begin{lemma}\label{supp:lem:half}
If $A_2\ge0$ or $A_3\ge0$, then $r>1/2$.
Consequently $v=2r-1>0$ throughout regimes II and III.
\end{lemma}
\begin{proof}
Both polynomials are strictly increasing in $r$, since
\[
\partial_r A_2=2r(1+u)^2+1>0,\qquad
\partial_r A_3=2r(u^4+2u^3+u^2+2u+1)+u^3+1>0.
\]
At $r=1/2$,
\[
A_2(1/2,u)=-\frac{1+u(2-u)}4<0,
\]
and
\[
A_3(1/2,u)=
-\frac{1+2u(1-u)+u^2(1-u^2)+4u^2(1-u)}4<0.
\]
Every term after the 1 in the latter numerator is nonnegative for
$0<u\le1$. Monotonicity proves the claim.
\end{proof}
The nontrivial denominator signs now have elementary certificates:
\begin{align*}
2H&=v(2u^2+5u+2)+2u^2+3u+2>0,\\
-G&=r^2u(1-u)+2r^2+2ru+3r+3>0.
\end{align*}
All other denominator factors are positive monomials or one of
$D_2,D_3,L,U,J_0$, each a sum of positive terms.
In particular, $H>0$ on the branch where it is used. Every fraction can
now be written over a strictly positive denominator, leaving only its
numerator sign to be established.

\section{Orienting the weight numerators}
We use the positive denominators from the previous section to specify
each probability. The table identifies the policies in each mixture and
the divisor associated with each numerator.

Let $N_j^{(k)}$ be the numerator of weight $j$ in branch $k$ after choosing
the positive denominators below. The strategy weight is $w_j^{(k)}=N_j^{(k)}/d_j^{(k)}$, where
$d_j^{(k)}$ is the strictly positive denominator listed below.
In branch III,B the denominator uses $-G>0$. Policies not listed for a
branch have weight zero. The expressions below define the numerators
directly; the equivalent expanded fractions are also supplied in the
code archive.
\begin{center}\small
\begin{tabular}{@{}lll@{}}
\toprule
Branch & Policy labels & Positive denominator\\ \midrule
$\mathrm{II}^{-}$ & $26,27,38$ & $HD_2$\\
 & $33,39$ & $rHD_2$\\
$\mathrm{II}^{+}$ & $26$ & $UD_2$\\
 & $33$ & $U$\\
 & $37,38,41$ & $ULD_2$\\
$\mathrm{III,A}$ & $2$ & $D_3$\\
 & $26$ & $UD_3$\\
 & $28,33,37,40$ & $2UrJ_0D_3$\\
$\mathrm{III,B}$ & $2,3,22,26,37,40$ & $(-G)D_3$\\
$\mathrm{III,C}$ & $2,5$ & $uD_3$\\
 & $28,33,37,40$ & $2ruJ_0D_3$\\
$\mathrm{III,D}$ & $5$ & $u^2D_3$\\
 & $19$ & $ru^2D_3$\\
 & $28,33,37,40$ & $2ru^2J_0D_3$\\ \bottomrule
\end{tabular}
\end{center}
For ten of the weights, the numerator is already nonnegative by the
branch conditions or positivity of the cost ratios:
\begin{align*}
N_{26}^{(\mathrm{II}^+)}&=A_2,&
N_{33}^{(\mathrm{II}^+)}&=u,&N_{41}^{(\mathrm{II}^+)}&=Q_2,\\
N_{2}^{(\mathrm{III,A})}&=B_3,&
N_{26}^{(\mathrm{III,A})}&=X,&N_{33}^{(\mathrm{III,A})}&=Y,\\
N_2^{(\mathrm{III,C})}&=(1+u)Z,&
N_5^{(\mathrm{III,C})}&=-X,\\
N_5^{(\mathrm{III,D})}&=A_3,&
N_{19}^{(\mathrm{III,D})}&=(1+r)(1+u)(-Z).
\end{align*}
These expressions settle ten of the $34$ listed weights. The remaining
$24$ have explicit decompositions, beginning with the five simpler ones
in the next section.

\section{Five elementary numerator decompositions}\label{sec:short}
Several numerators can be written using elementary nonnegative factors
such as $1-r$, $1-u$ and $1-ru$. Five such expressions are given below;
only the second uses an additional branch inequality.
\begin{align*}
N_{38}^{(\mathrm{II}^-)}
 &=2r(1-r+r^2)+u(1-r^2+7r^3)\\
 &\quad+ru(1-u)\bigl(2(1-r^2)+r\bigr),\\
N_{37}^{(\mathrm{II}^+)}
 &=(1-r)A_2+ru(ru^2+2ru+2u+2),\\
N_{38}^{(\mathrm{II}^+)}
 &=r^3(1+u)^3+r^2u+u(1-ru)+1,\\
N_{37}^{(\mathrm{III,B})}
 &=4+r(3-r^3)+4r^2(1-r^2u)+r^2u(2-r)+2ru\\
 &\quad+r^3u^3(2-r)+2r^2u^2(1-r)(2r+1),\\
N_{40}^{(\mathrm{III,B})}
 &=4+2r(1-r+r^2)+r^4u^4+2r^2u^2\\
 &\quad+2r^3(r+1)u^2(1-u)+13r^3u(1-u)\\
 &\quad+ru\bigl[5+(1-r)\{3+(1-r)(7r+3)\}\bigr].
\end{align*}
All terms are nonnegative for $0<r,u\le1$, with $A_2\ge0$ used only
in the second identity. The factor $1-r+r^2$ equals $(1-r)+r^2$.
We have now proved the signs of fifteen weights; the other nineteen
use the defining polynomials of their branches more directly.

\section{Examples of branchwise identities}
The region inequalities can be used as factors in a numerator. For
example, a term containing $K_2$ is nonnegative in regime II provided
its multiplier is nonnegative there. The following three expressions
illustrate this construction before the complete tables.
\begin{align*}
N_{26}^{(\mathrm{II}^-)}
 &=r(1+u)^2A_2+
   \bigl[r(1+u-u^2)+3-u^2\bigr]K_2+(1-u)(-Q_2),\\
N_{37}^{(\mathrm{III,A})}
 &=Y+2(1+r^2)B_3+2r(1-r)X,\\
N_{40}^{(\mathrm{III,A})}
 &=Y+2\bigl[1+r+ru(1-r)\bigr]B_3+4r^2uX.
\end{align*}
The coefficient in front of $K_2$ is positive because
$1+u-u^2=1+u(1-u)\ge1$ and $3-u^2\ge2$.
Each term has the required sign in its own branch. These three identities
are included in the nineteen complete certificates below, written there
using a uniform notation for their coefficients.

\section{Complete branchwise numerator formulas}\label{sec:tables}
We now give the nineteen remaining numerator formulas. Each is expressed
as a sum of branch polynomials multiplied by nonnegative coefficients,
with a possible additional nonnegative remainder. This representation
lets the sign of each term be read from the branch table.

In this section only, put
\[
s=1-r,\qquad t=1-u,\qquad v=2r-1.
\]
Thus $r,u,s,t$ are nonnegative and $v>0$ on all branches under
consideration. The symbols $R_0,R_1,\ldots$ are local to each displayed
identity and are polynomials with nonnegative integer coefficients in
$r,u,s,t$. The integer multiplier on the left is at most 4.
Every factor multiplying an $R_i$ is either a defining nonnegative
region polynomial, $v$, or the already proved positive factor $H$.
Consequently each displayed identity proves its numerator nonnegative,
including points on the region boundaries. The following five subsections
cover all remaining cases; the weights of branch II$^+$ were already
settled above.

\subsection{Branch \texorpdfstring{$\mathrm{II}^{-}$}{II-}}
When $Q_2\le0$ in regime II, policy~38 is covered by the elementary
decomposition. For the other four policies the available factors are
$A_2,K_2,-Q_2$, together with $v$ and $H$ where needed.

\Needspace{9\baselineskip}
\paragraph{Policy 26.}
\[N_{26}^{(\mathrm{II}^{-})}=R_{0}A_2 + R_{1}K_2 + R_{2}(-Q_2).\]
\begin{align*}
R_{0}&=r u^{2} + 2 r u + r,\\
R_{1}&=r t u + r + t u + t + 2,\\
R_{2}&=t.
\end{align*}
\Needspace{9\baselineskip}
\paragraph{Policy 27.}
\[N_{27}^{(\mathrm{II}^{-})}=R_{0} + R_{1}A_2 + R_{2}K_2 + R_{3}(-Q_2).\]
\begin{align*}
R_{0}&=r s u^{2} + r s^{2} u + r s^{2} + s u^{2} + t u + u + s,\\
R_{1}&=s,\\
R_{2}&=t,\\
R_{3}&=1.
\end{align*}
\Needspace{9\baselineskip}
\paragraph{Policy 33.}
\[N_{33}^{(\mathrm{II}^{-})}=R_{0} + R_{1}K_2 + R_{2}(-Q_2) + R_{3}v + R_{4}H.\]
\begin{align*}
R_{0}&=r^{3} t u^{2} + r^{3} s t + r^{3} t^{3} + r^{2} u^{2}\\
&\quad+ 2 r^{2} t u + r^{2} u + r u + r s^{2} + s^{3} u + s^{4} + 1,\\
R_{1}&=u,\\
R_{2}&=1,\\
R_{3}&=r s^{2} u^{2} + r s u^{2} + r t u + r u + s,\\
R_{4}&=s^{2}.
\end{align*}
\Needspace{9\baselineskip}
\paragraph{Policy 39.}
\[N_{39}^{(\mathrm{II}^{-})}=R_{0} + R_{1}A_2 + R_{2}K_2 + R_{3}(-Q_2).\]
\begin{align*}
R_{0}&=r^{4} u^{2} + r^{3} t u^{2} + r^{2} t u^{2} + r u + r s^{2} + u\\
&\quad+ s,\\
R_{1}&=2 r s,\\
R_{2}&=2,\\
R_{3}&=s.
\end{align*}
With policy~38, these four formulas give a nonnegative weight for every
listed policy in branch II$^-$. Their divisors are $HD_2$ or $rHD_2$,
which are positive by the preceding lemma.

\subsection{Branch \texorpdfstring{$\mathrm{III,A}$}{III, A}}
Case A of regime III has $Y,X\ge0$. The weights of policies~2, 26 and
33 use $B_3,X,Y$ directly. The other three numerators are given by the
following expressions.

\Needspace{9\baselineskip}
\paragraph{Policy 28.}
\[N_{28}^{(\mathrm{III,A})}=R_{0} + R_{1}B_3 + R_{2}X.\]
\begin{align*}
R_{0}&=r^{3} s + r^{2} s t u + 2 r^{2} s t + r s^{3} + r s t + r s\\
&\quad+ s^{3},\\
R_{1}&=4 r^{2} u,\\
R_{2}&=2 r^{2}.
\end{align*}
\Needspace{9\baselineskip}
\paragraph{Policy 37.}
\[N_{37}^{(\mathrm{III,A})}=R_{0}B_3 + R_{1}Y + R_{2}X.\]
\begin{align*}
R_{0}&=2 r^{2} + 2,\\
R_{1}&=1,\\
R_{2}&=2 r s.
\end{align*}
\Needspace{9\baselineskip}
\paragraph{Policy 40.}
\[N_{40}^{(\mathrm{III,A})}=R_{0}B_3 + R_{1}Y + R_{2}X.\]
\begin{align*}
R_{0}&=2 r s u + 2 r + 2,\\
R_{1}&=1,\\
R_{2}&=4 r^{2} u.
\end{align*}
Thus the six listed weights in case A are nonnegative; all four longer
fractions in this branch use the same positive denominator
$2UrJ_0D_3$.

\subsection{Branch \texorpdfstring{$\mathrm{III,B}$}{III, B}}
For $Y\le0$, case B uses a different set of search policies. The common
denominator is $(-G)D_3$. Policies~37 and 40 have already been treated by
the elementary decompositions; the following numerators use $A_3,B_3$
and $-Y$ as nonnegative factors.

\Needspace{9\baselineskip}
\paragraph{Policy 2.}
\[N_{2}^{(\mathrm{III,B})}=R_{0} + R_{1}A_3 + R_{2}B_3 + R_{3}v.\]
\begin{align*}
R_{0}&=r^{4} t u^{2} + r^{2} s t u^{2} + r^{2} s^{2} u + r^{2} s u\\
&\quad+ r^{2} s + 2 u + 2 s + 1,\\
R_{1}&=r s,\\
R_{2}&=r s t u + u + 3,\\
R_{3}&=r s u^{2} + u + s.
\end{align*}
\Needspace{9\baselineskip}
\paragraph{Policy 3.}
\[2N_{3}^{(\mathrm{III,B})}=R_{0} + R_{1}A_3 + R_{2}B_3 + R_{3}(-Y).\]
\begin{align*}
R_{0}&=r^{6} t u + r^{5} u^{2} + r^{5} s t u + r^{5} u + r^{5} t^{3}\\
&\quad+ r^{4} s t u^{2} + r^{4} t^{3} u + r^{4} t u + r^{4} s t\\
&\quad+ r^{3} s^{2} u^{3} + 2 r^{3} s^{2} u^{2} + r^{3} t^{2} u\\
&\quad+ 2 r^{3} t^{2} + r^{2} u + r s u^{6} + r s t u^{5} + r s^{3} u^{4}\\
&\quad+ r t u^{4} + r s^{4} u^{3} + r t u^{3} + r s t u^{2} + r t^{2} u\\
&\quad+ r t u + 2 r t + r + s^{2} u^{4} + s^{3} u^{3} + t u^{3}\\
&\quad+ s t u^{2} + s^{3} u,\\
R_{1}&=s u,\\
R_{2}&=r t^{2} u + 2 r s^{2},\\
R_{3}&=s + 3.
\end{align*}
\Needspace{9\baselineskip}
\paragraph{Policy 22.}
\[2N_{22}^{(\mathrm{III,B})}=R_{0} + R_{1}A_3 + R_{2}B_3 + R_{3}(-Y) + R_{4}v.\]
\begin{align*}
R_{0}&=r^{4} t^{3} u + r^{4} t + r^{3} + r^{2} s^{2} + r^{2} + r t u\\
&\quad+ r u + r,\\
R_{1}&=r s + 2 r,\\
R_{2}&=3 r t u + r s t,\\
R_{3}&=1,\\
R_{4}&=2 r^{3} + r^{2} s u^{2} + r s u + r t u + r + 1.
\end{align*}
\Needspace{9\baselineskip}
\paragraph{Policy 26.}
\[4N_{26}^{(\mathrm{III,B})}=R_{0} + R_{1}A_3 + R_{2}B_3 + R_{3}(-Y) + R_{4}v.\]
\begin{align*}
R_{0}&=r^{4} t u^{3} + r^{4} t u + 2 r^{4} + r^{3} s u^{4} + r^{3} t\\
&\quad+ r^{3} + r^{2} t^{2} u + r^{2} t u + r^{2} t + r t u + r t\\
&\quad+ s^{3} t u^{2} + 16 t,\\
R_{1}&=r s,\\
R_{2}&=r^{2} t^{2} + 3 r^{2} t + r s u + r s + 5 r t,\\
R_{3}&=u + 2,\\
R_{4}&=4 r^{3} t u + 9 r^{3} t + 14 r^{2} t + r t^{2} u + 2 r t u\\
&\quad+ 14 r t + t u^{2} + 17 t + 1.
\end{align*}
All six fractions in case B now have nonnegative numerators over the
same positive denominator. The argument also applies at $Y=0$.

\subsection{Branch \texorpdfstring{$\mathrm{III,C}$}{III, C}}
When $Y\ge0$ and $X\le0$, the sign of $Z$ distinguishes the final two
mixtures. In case C, $Z\ge0$ and the weights of policies~2 and 5 are
immediate. The following four expressions complete the construction
for this sign of $Z$.

\Needspace{9\baselineskip}
\paragraph{Policy 28.}
\[2N_{28}^{(\mathrm{III,C})}=R_{0} + R_{1}A_3 + R_{2}B_3 + R_{3}(-X) + R_{4}Z.\]
\begin{align*}
R_{0}&=r^{4} t^{3} u + r^{3} s u^{4} + r^{3} t u^{3} + r^{3} s t u\\
&\quad+ r^{2} s^{2} u^{2} + r^{2} s t^{2} u + r^{2} t^{2} u + r^{2} t u\\
&\quad+ r s^{3} u + s^{2} t u^{2} + s u^{2} + t^{2} u,\\
R_{1}&=r^{2},\\
R_{2}&=r^{2} u^{2},\\
R_{3}&=r s + r t + t + 1,\\
R_{4}&=r u.
\end{align*}
\Needspace{9\baselineskip}
\paragraph{Policy 33.}
\[2N_{33}^{(\mathrm{III,C})}=R_{0} + R_{1}B_3 + R_{2}Y + R_{3}(-X).\]
\begin{align*}
R_{0}&=r^{3} s t^{2} + r^{2} s^{2} t u^{2} + r s t u + 2 r s t + s t,\\
R_{1}&=2 r t^{2},\\
R_{2}&=u,\\
R_{3}&=2 r^{2} t u + r^{2} t + r t + r + 1.
\end{align*}
\Needspace{9\baselineskip}
\paragraph{Policy 37.}
\[2N_{37}^{(\mathrm{III,C})}=R_{0} + R_{1}B_3 + R_{2}(-X) + R_{3}Z.\]
\begin{align*}
R_{0}&=2 r^{3} t u + 2 r^{2} + 2 r,\\
R_{1}&=r t^{2} + 2 r t + r + s^{2},\\
R_{2}&=r s t,\\
R_{3}&=r^{2} t + r^{2} + r t u + 3 r + 1.
\end{align*}
\Needspace{9\baselineskip}
\paragraph{Policy 40.}
\[4N_{40}^{(\mathrm{III,C})}=R_{0} + R_{1}B_3 + R_{2}(-X) + R_{3}Z + R_{4}v.\]
\begin{align*}
R_{0}&=r^{4} t u + r^{4} t^{4} + r^{3} t u^{2} + r^{3} s u + r^{3} u\\
&\quad+ r^{2} u^{2} + r^{2} u + r u + r s t + 3,\\
R_{1}&=r^{2} t u + r^{2} t + r^{2} + 4 r t + 2 r,\\
R_{2}&=r^{2} t,\\
R_{3}&=2 r + 7,\\
R_{4}&=r^{2} s t^{2} + r u.
\end{align*}
All six weights in case C are consequently nonnegative, including when
$Z=0$. For the other sign of $Z$, policy~19 replaces policy~2.

\subsection{Branch \texorpdfstring{$\mathrm{III,D}$}{III, D}}
Case D covers $Z\le0$ when $Y\ge0$ and $X\le0$. Policies~5 and 19
have numerators proportional to $A_3$ and $-Z$, respectively. For the
remaining policies we use the following decompositions.

\Needspace{9\baselineskip}
\paragraph{Policy 28.}
\[2N_{28}^{(\mathrm{III,D})}=R_{0} + R_{1}B_3 + R_{2}(-X) + R_{3}(-Z) + R_{4}v.\]
\begin{align*}
R_{0}&=r^{4} u^{4} + r^{3} t^{2} u^{2} + r^{2} u^{4} + r^{2} s^{2} u^{2}\\
&\quad+ r s^{3} u^{3} + s u^{2},\\
R_{1}&=r^{2} u^{2},\\
R_{2}&=r u,\\
R_{3}&=2 r^{2} + 3 r t u + 2 r + 2 u,\\
R_{4}&=r^{3} u^{4} + r^{3} t^{2} u^{2} + r^{2} u^{4} + r s^{2} u^{3}\\
&\quad+ s^{2} u^{2}.
\end{align*}
\Needspace{9\baselineskip}
\paragraph{Policy 33.}
\[2N_{33}^{(\mathrm{III,D})}=R_{0} + R_{1}A_3 + R_{2}B_3 + R_{3}(-X) + R_{4}(-Z) + R_{5}v.\]
\begin{align*}
R_{0}&=r^{3} t u^{2} + r^{3} u,\\
R_{1}&=3 r + 2,\\
R_{2}&=2 r^{2} t u + 2 r^{2} t + 2 r t u + 2 t u + 1,\\
R_{3}&=s,\\
R_{4}&=r u^{2},\\
R_{5}&=2 r^{3} + r^{2} u^{2} + r^{2} t + r^{2} + r u + 4 r + 2 t.
\end{align*}
\Needspace{9\baselineskip}
\paragraph{Policy 37.}
\[4N_{37}^{(\mathrm{III,D})}=R_{0} + R_{1}A_3 + R_{2}B_3 + R_{3}Y + R_{4}(-X) + R_{5}(-Z) + R_{6}v.\]
\begin{align*}
R_{0}&=r^{3} u^{2} + r^{3} u + r^{2} t u^{3} + r^{2} u + r s t u^{2}\\
&\quad+ r u^{2},\\
R_{1}&=r s + s,\\
R_{2}&=r t u + t u,\\
R_{3}&=u,\\
R_{4}&=2 r^{2} t + r t^{2} + 3 r t,\\
R_{5}&=r^{2} u + r^{2} t^{2} + t^{2},\\
R_{6}&=r^{3} u^{3} + u^{2} + s^{2} u.
\end{align*}
\Needspace{9\baselineskip}
\paragraph{Policy 40.}
\[2N_{40}^{(\mathrm{III,D})}=R_{0} + R_{1}A_3 + R_{2}B_3 + R_{3}(-Z) + R_{4}v.\]
\begin{align*}
R_{0}&=r^{3} t u^{3} + r^{3} t^{3} u + r^{2} t u^{2} + r^{2} t u\\
&\quad+ s^{3} u^{3} + s t u^{2} + t u + u,\\
R_{1}&=2 r^{2} + 2 r + 1,\\
R_{2}&=2 r^{2} t u + 5 r t u,\\
R_{3}&=t^{2},\\
R_{4}&=r^{3} t u^{3} + r^{3} t u + r^{2} s u + r s u^{2} + u.
\end{align*}
Case D completes the list of numerator formulas. Together with the
positive denominators and the earlier decompositions, these formulas
establish the nonnegativity of all $34$ weights.

\section{Normalization and expected costs}
The weights have now been specified with their signs established. To see
that the six mixtures attain the claimed values, we also need unit total
weight and the payoff identities against the individual placements. This
section collects these properties.

\begin{proposition}
For each of the six Searcher mixtures, every listed weight is
nonnegative throughout the corresponding parameter branch.
\end{proposition}
\begin{proof}
Lemma~\ref{supp:lem:half} justifies $v>0$ and the displayed denominator
factorizations justify all denominator signs. Ten numerator signs follow
from the region definitions, five from Section~\ref{sec:short}, and
the other 19 from Section~\ref{sec:tables}. These cases account for all
weights in the support lists. All inequalities are non-strict except
for the denominator factors, so the argument includes the shared
region boundaries.
\end{proof}

For branch $k$, let $\overline C^{(k)}(x)$ denote the expected cost
against placement $x$. With the inspection counts of the named policies
from Appendix A of the main paper, multiply each policy cost by its weight
and sum over the listed policies. Simplifying the rational expressions
gives
\begin{align*}
\sum_j w_j^{(k)}&=1 &&\text{for each of the six branches},\\
\overline C^{(k)}(x)&=V_2
 &&(k=\mathrm{II}^{\pm},\ x\in\{011,020,101,110,200\}),\\
\overline C^{(k)}(002)-V_2&=-\frac{(1+r)K_2}{D_2}
 &&(k=\mathrm{II}^{\pm}),\\
\overline C^{(k)}(x)&=V_3
 &&(k=\mathrm{III,A},\ldots,\mathrm{III,D},\ x\in\mathcal H).
\end{align*}
The first equality makes each set of nonnegative weights a probability
distribution. Since $K_2\ge0$ in regime II, the next two identities bound
the expected cost of every placement by $V_2$; the last does the same
with $V_3$ in regime III. These are precisely the upper bounds required
by the minimax proof in the main paper.

\clearpage
\section{All 42 opening-count profiles}
The preceding formulas use only the fourteen named policies. To specify
the entire finite game, we also list every distinct opening-count profile.
This catalogue allows the reduction and both sides of the minimax argument
to be checked directly.

The profiles below are listed in lexicographic order, with placements
ordered as $002$, $011$, $020$, $101$, $110$, $200$. Each entry $(i,j,k)$ means $i$ inspections
of box $A$, $j$ of $B$ and $k$ of $C$ against the corresponding placement.
The catalogue row numbers $1,\ldots,42$ are distinct from the named
policy labels in Appendix A.
\begingroup\footnotesize\renewcommand{\arraystretch}{0.98}
\begin{longtable}{@{}rcccccc@{}}
\caption{The complete six-placement opening-count-profile matrix.}\\
\toprule
No. & $002$ & $011$ & $020$ & $101$ & $110$ & $200$\\
\midrule
\endfirsthead
\toprule
No. & $002$ & $011$ & $020$ & $101$ & $110$ & $200$\\
\midrule
\endhead
1 & $(0,0,2)$ & $(0,1,2)$ & $(0,2,1)$ & $(1,1,2)$ & $(1,2,1)$ & $(2,1,1)$\\
2 & $(0,0,2)$ & $(0,1,2)$ & $(1,2,1)$ & $(1,1,2)$ & $(1,1,1)$ & $(2,1,1)$\\
3 & $(0,0,2)$ & $(0,1,2)$ & $(1,2,1)$ & $(1,1,2)$ & $(2,1,1)$ & $(2,0,1)$\\
4 & $(0,0,2)$ & $(1,1,2)$ & $(0,2,1)$ & $(1,0,2)$ & $(1,2,1)$ & $(2,1,1)$\\
5 & $(0,0,2)$ & $(1,1,2)$ & $(1,2,1)$ & $(1,0,2)$ & $(1,1,1)$ & $(2,1,1)$\\
6 & $(0,0,2)$ & $(1,1,2)$ & $(1,2,1)$ & $(1,0,2)$ & $(2,1,1)$ & $(2,0,1)$\\
7 & $(0,1,2)$ & $(0,1,1)$ & $(0,2,1)$ & $(1,1,2)$ & $(1,2,1)$ & $(2,1,1)$\\
8 & $(0,1,2)$ & $(0,1,1)$ & $(1,2,1)$ & $(1,1,2)$ & $(1,1,1)$ & $(2,1,1)$\\
9 & $(0,1,2)$ & $(0,1,1)$ & $(1,2,1)$ & $(1,1,2)$ & $(2,1,1)$ & $(2,0,1)$\\
10 & $(0,1,2)$ & $(0,2,1)$ & $(0,2,0)$ & $(1,1,2)$ & $(1,2,1)$ & $(2,1,1)$\\
11 & $(0,1,2)$ & $(1,1,1)$ & $(1,2,1)$ & $(1,1,2)$ & $(1,1,0)$ & $(2,1,1)$\\
12 & $(0,1,2)$ & $(1,2,1)$ & $(0,2,0)$ & $(1,1,2)$ & $(1,2,0)$ & $(2,1,1)$\\
13 & $(0,1,2)$ & $(1,2,1)$ & $(1,2,0)$ & $(1,1,2)$ & $(1,1,0)$ & $(2,1,1)$\\
14 & $(1,0,2)$ & $(1,1,2)$ & $(0,2,1)$ & $(1,0,1)$ & $(1,2,1)$ & $(2,1,1)$\\
15 & $(1,0,2)$ & $(1,1,2)$ & $(1,2,1)$ & $(1,0,1)$ & $(1,1,1)$ & $(2,1,1)$\\
16 & $(1,0,2)$ & $(1,1,2)$ & $(1,2,1)$ & $(1,0,1)$ & $(2,1,1)$ & $(2,0,1)$\\
17 & $(1,0,2)$ & $(1,1,2)$ & $(1,2,1)$ & $(1,1,1)$ & $(1,1,0)$ & $(2,1,1)$\\
18 & $(1,0,2)$ & $(1,1,2)$ & $(1,2,1)$ & $(2,0,1)$ & $(2,1,1)$ & $(2,0,0)$\\
19 & $(1,0,2)$ & $(1,1,2)$ & $(1,2,1)$ & $(2,1,1)$ & $(1,1,0)$ & $(2,1,0)$\\
20 & $(1,0,2)$ & $(1,1,2)$ & $(1,2,1)$ & $(2,1,1)$ & $(2,1,0)$ & $(2,0,0)$\\
21 & $(1,1,2)$ & $(0,1,1)$ & $(0,2,1)$ & $(1,1,1)$ & $(1,2,1)$ & $(2,1,1)$\\
22 & $(1,1,2)$ & $(0,1,1)$ & $(0,2,1)$ & $(2,1,1)$ & $(1,2,1)$ & $(2,1,0)$\\
23 & $(1,1,2)$ & $(0,1,1)$ & $(1,2,1)$ & $(1,1,1)$ & $(1,1,1)$ & $(2,1,1)$\\
24 & $(1,1,2)$ & $(0,1,1)$ & $(1,2,1)$ & $(1,1,1)$ & $(2,1,1)$ & $(2,0,1)$\\
25 & $(1,1,2)$ & $(0,1,1)$ & $(1,2,1)$ & $(2,1,1)$ & $(1,1,1)$ & $(2,1,0)$\\
26 & $(1,1,2)$ & $(0,2,1)$ & $(0,2,0)$ & $(1,1,1)$ & $(1,2,1)$ & $(2,1,1)$\\
27 & $(1,1,2)$ & $(0,2,1)$ & $(0,2,0)$ & $(2,1,1)$ & $(1,2,1)$ & $(2,1,0)$\\
28 & $(1,1,2)$ & $(1,1,1)$ & $(0,2,1)$ & $(1,0,1)$ & $(1,2,1)$ & $(2,1,1)$\\
29 & $(1,1,2)$ & $(1,1,1)$ & $(1,2,1)$ & $(1,0,1)$ & $(1,1,1)$ & $(2,1,1)$\\
30 & $(1,1,2)$ & $(1,1,1)$ & $(1,2,1)$ & $(1,0,1)$ & $(2,1,1)$ & $(2,0,1)$\\
31 & $(1,1,2)$ & $(1,1,1)$ & $(1,2,1)$ & $(1,1,1)$ & $(1,1,0)$ & $(2,1,1)$\\
32 & $(1,1,2)$ & $(1,1,1)$ & $(1,2,1)$ & $(2,0,1)$ & $(2,1,1)$ & $(2,0,0)$\\
33 & $(1,1,2)$ & $(1,1,1)$ & $(1,2,1)$ & $(2,1,1)$ & $(1,1,0)$ & $(2,1,0)$\\
34 & $(1,1,2)$ & $(1,1,1)$ & $(1,2,1)$ & $(2,1,1)$ & $(2,1,0)$ & $(2,0,0)$\\
35 & $(1,1,2)$ & $(1,2,1)$ & $(0,2,0)$ & $(1,1,1)$ & $(1,2,0)$ & $(2,1,1)$\\
36 & $(1,1,2)$ & $(1,2,1)$ & $(0,2,0)$ & $(2,1,1)$ & $(1,2,0)$ & $(2,1,0)$\\
37 & $(1,1,2)$ & $(1,2,1)$ & $(1,2,0)$ & $(1,0,1)$ & $(1,1,1)$ & $(2,1,1)$\\
38 & $(1,1,2)$ & $(1,2,1)$ & $(1,2,0)$ & $(1,0,1)$ & $(2,1,1)$ & $(2,0,1)$\\
39 & $(1,1,2)$ & $(1,2,1)$ & $(1,2,0)$ & $(1,1,1)$ & $(1,1,0)$ & $(2,1,1)$\\
40 & $(1,1,2)$ & $(1,2,1)$ & $(1,2,0)$ & $(2,0,1)$ & $(2,1,1)$ & $(2,0,0)$\\
41 & $(1,1,2)$ & $(1,2,1)$ & $(1,2,0)$ & $(2,1,1)$ & $(1,1,0)$ & $(2,1,0)$\\
42 & $(1,1,2)$ & $(1,2,1)$ & $(1,2,0)$ & $(2,1,1)$ & $(2,1,0)$ & $(2,0,0)$\\
\bottomrule
\end{longtable}
\endgroup
Taking the scalar product of each triple with $(a,b,c)$ and transposing
the catalogue gives the six-row, $42$-column payoff matrix. It contains
the payoffs of every non-wasteful deterministic search, and the weights
above specify mixtures of its columns.

\end{document}